\documentclass[12pt]{article}
\usepackage{cite}

\usepackage{latexsym,amssymb,amsthm,amsmath}
\usepackage{graphicx}

\theoremstyle{plain}
\newtoks\thehProclaim
\newtheorem*{Proclaim}{\the\thehProclaim}

\newtheorem{proposition}{Proposition}

\newtheorem{theorem}{Theorem}
\newtheorem{lemma}{Lemma}

\theoremstyle{definition}
\newtheorem{definition}{Definition}

\newcommand{\eps}{\varepsilon}

\newcommand{\tow}{\mathcal{R}}
\newcommand{\mapp}{M_{\rm{ap}}}

\newcommand{\si}{\mathfrak{s}}

\newcommand{\orb}{\mathrm{Orb}}
\newcommand{\orbp}{{\mathrm{Orb}_+}}

\newcommand{\T}{\mathcal{T}}

\newcommand{\OT}{\mathcal{OT}}
\newcommand{\OTP}{\mathcal{OTP}}
\newcommand{\ot}{\mathrm{ot}}
\newcommand{\otp}{\mathrm{otp}}

\newcommand{\uag}{UA}

\title{On a universal Borel adic space\thanks{This work is supported by the Program of the Presidium of the Russian
Academy of Sciences No.~01 ``Fundamental Mathematics and its Applications''
under grant PRAS-18-01.}}

\author{A. M. Vershik\thanks{St.Petersburg Department of Steklov Institute of Mathematics
and St.Petersburg State University, St.Petersburg, Russia.
E-mail: avershik@pdmi.ras.ru.} \and P. B. Zatitskii\thanks{St.Petersburg State University and St.Petersburg Department of Steklov Institute of Mathematics, St.Petersburg, Russia. E-mail:
pavelz@pdmi.ras.ru.}}

\date{September 26, 2018}

\begin{document}

\maketitle

\begin{abstract}{We prove that the so-called uniadic graph and its adic automorphism
are Borel universal, i.e., every aperiodic Borel automorphism is
isomorphic to the restriction of this automorphism to a subset
invariant under the adic transformation, the isomorphism being defined
on a universal (with respect to the measure) set. We develop the
concept of basic filtrations and combinatorial definiteness of
automorphisms suggested in our previous paper. Bibliography: $10$ titles.}
 \end{abstract}

\section{Introduction}
In~\cite{V81, V82}, the first author proved that every ergodic automorphism of a Lebesgue space has an adic realization, i.e., is isomorphic to the adic shift on the path space of some graded graph equipped with a central measure. In~\cite{VZ18}, it is proved that for such a graph one can always take the
so-called uniadic graph~${\rm{\uag}}$
(see~Sec.~\ref{sec2}), varying only a central measure
on its path space. The purpose of this note is to prove
a Borel analog of this result: every aperiodic Borel
automorphism of a separable metric space can be realized in the path space of the uniadic graph. The proof  is based on the construction of a so-called basic Borel filtration of a given automorphism. For more details on the history of the problem, see~\cite{V17,VZ18}.

The relation between the Borel and metric approaches in dynamics and representation theory has been considered in a number of papers, see~\cite{Kech,Kan,Tho,Sch}. One of the problems linking both approaches is to describe all invariant measures for a given Borel automorphism or a Borel filtration.

The natural question, which is solved in the affirmative in this paper, is whether one can define a universal automorphism in a standard Borel space and an approximation of this automorphism so that every Borel automorphism of a standard Borel space is isomorphic (up to a set of zero measure for all aperiodic measures) to the restriction of this automorphism to an invariant subset, see Theorem~\ref{thborel}. It turns out that such a Borel space is the path space of the uniadic (= universal + adic) graph we define below, the desired automorphism is the corresponding adic shift, and the approximation is determined by the tail filtration
of the graph. The proof uses the idea of a paper treating the old
and simpler question about the Borel universality of Rokhlin's lemma.

The corollaries obtained in this paper and in~\cite{VZ18} concern
the theory of uniform approximation of actions of the
group~$\mathbb Z$, in particular, a new method of encoding automorphisms via filtrations.

Let us describe the setting of the problem in more detail. Let $X$ be a standard Borel space. We say that a map $T \colon X \to X$ is a \emph{Borel automorphism} of~$X$ if it is invertible and both $T$ and~$T^{-1}$ are Borel measurable. By $\mapp(X,T)$ we denote the space of all $T$-invariant aperiodic probability measures on~$X$. We say that a Borel subset $\widehat  X \subset X$ is  \emph{metrically universal} if $\mu(\widehat  X )=1$ for every measure $\mu \in \mapp(X,T)$.

The main result of this note is the following theorem.

\begin{theorem}[Borel universality of the uniadic graph]\label{thborel}
Let $T$ be an aperiodic Borel automorphism of a separable
metric space~$X$. Then there exists a metrically universal Borel
subset $\widehat  X \subset X$ and a Borel measurable
injective embedding of $\widehat  X$ into the path space~$\T{\rm(\uag)}$
 of the uniadic graph~$\rm{\uag}$ that sends the automorphism~$T$
to the adic shift on~$\T{\rm(\uag)}$.
\end{theorem}

The proof of Theorem~\ref{thborel} uses several different ideas.
The first one consists in obtaining a Borel version of
Rokhlin's lemma (see~\cite{GV06}). The second idea is to
iterate a weakened version of Rokhlin's lemma in order to construct a basic Borel filtration and an adic realization of the automorphism (see~\cite{V81,V82} and~\cite{VZ18}). We will prove a weakened Borel version of Rokhlin's lemma (Lemma~\ref{lem1}), iterate it to construct a basic Borel filtration of the automorphism (Theorem~\ref{th1}),
and prove Theorem~\ref{thborel}.

\section{Borel filtrations}
In~\cite{VZ18}, we studied measurable partitions of a Lebesgue space and filtrations (decreasing sequences of measurable partitions\footnote{In a decreasing sequence of partitions, elements of partitions  become coarser. The orderings of partitions adopted in combinatorics and in measure theory are reverse to each other; we use the terminology of measure theory and functional analysis.}) on Lebesgue spaces. In this paper, the main notions are carried over to the Borel case.

\noindent {\bf~2.1. Basic filtrations, colored filtrations}

\begin{definition}
A Borel filtration of a standard Borel space is a decreasing sequence of Borel partitions $\Xi =\{ \xi_n\}_{n \geq 0}$ where $\xi_0$ is the partition into separate points.

A filtration $\Xi$ is said to be \emph{locally finite}\footnote{Note that this notion of local finiteness of a Borel filtration is, in general, different from that adopted in the metric theory of measurable partitions.} if for every~$n$ the sizes of the elements of the partition~$\xi_n$ are uniformly bounded by a constant, possibly depending on~$n$.

We say that $\Xi$ is an \emph{ordered} filtration if each element of the quotient partition $\xi_{n+1}/\xi_n$ is endowed with a measurable linear order (measurability means that the set of all points with given number in the elements of the partition~$\xi_{n+1}/\xi_n$ is Borel measurable); these orders induce a coherent order on the elements of the partitions~$\xi_n$, and hence on the classes of the limiting partition~$\bigcap_n \xi_n$ (which is not, in general, measurable); we assume that the order type is~$\mathbb{Z}$ for almost all classes.

A \emph{basic} filtration is a locally finite ordered filtration.
\end{definition}

\begin{definition}
Let $T$ be a Borel automorphism of a standard Borel space~$X$ and $\Xi = \{\xi_n\}_{n\geq 0}$ be a basic Borel filtration on~$X$. We say that $\Xi$ is a \emph{basic filtration for}~$T$ if the limiting partition $\bigcap_n \xi_n$ is the partition into the orbits of~$T$ and the order of~$\Xi$ is determined by~$T$, i.e.,
\begin{itemize}
\item every element $\alpha$ of the partition~$\xi_k$, $k \in \mathbb{N}$, is a finite orbit of a point $x\in X $ under~$T$:
$$\alpha = \{T^j x\colon j=0,\dots, |\alpha|-1\};$$
\item  for every $x \in X $, the union of all elements~$\xi_k(x)$ of~$\xi_k$ containing~$x$ coincides with the orbit~$\orb_T(x)$ of~$x$ under~$T$.
\end{itemize}
\end{definition}

\begin{definition}\label{defcolor}
Let $\xi$ be a Borel partition of a Borel space~$X$. We say that $\xi$ is a \emph{colored partition} if the quotient partition~$X/\xi$ is endowed with a finite Borel partition~$c[\xi]$ determining the \emph{colors} of the elements of~$\xi$. The partition~$c[\xi]$ will be called the \emph{coloring} of the partition~$\xi$.

A Borel filtration $\Xi = \{\xi_n\}_{n \geq 0}$ is said to be \emph{colored} if each partition~$\xi_n$ is endowed with a coloring~$c[\xi_n]$.
\end{definition}

\subsection{Combinatorial definiteness of a Borel filtration}

As in~\cite{VZ18}, we introduce the notion of combinatorial definiteness of a basic Borel filtration.

Recall the construction of the finite tree describing the structure
of a finite ordered filtration on a finite set. Let $A$
be an arbitrary finite set and $\{\eta_i\}_{i=0}^{n}$ be a
finite ordered filtration on~$A$ with the last partition~$\eta_n$
being trivial (consisting of a single nonempty class).
We construct an ordered graded tree corresponding to this
finite filtration as follows. The vertices of level~$i$ in
this tree correspond to the elements of the partition~$\eta_i$.
A vertex of level~$i+1$ is joined by an edge with a vertex
of level~$i$ if the corresponding elements of partitions are
nested. The $n$th level consists of a single vertex, and the
vertices of level~$0$ are the elements of the set~$A$. The set~$A$
is endowed with a linear order: the filtration order determines
an order on the edges leading from every vertex to vertices of
the previous level. The obtained graded tree will be called
the \textit{filtration tree on the set~$A$} (see Fig.~1). 
The set of all ordered graded finite trees will be denoted by~$\OT$. Besides, we consider trees with a marked vertex (leaf). The set of all ordered graded finite trees with a marked leaf will be denoted by~$\OTP$. In the case of a colored filtration, the coloring can be carried over to trees in a natural way: a
vertex of the tree is colored as the corresponding element of the partition.

\begin{figure}[ht]%
\begin{center}
\includegraphics[width=0.4\textwidth]{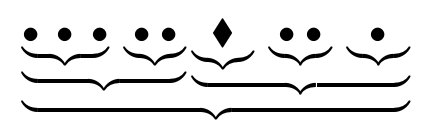}%
\includegraphics[width=0.4\textwidth]{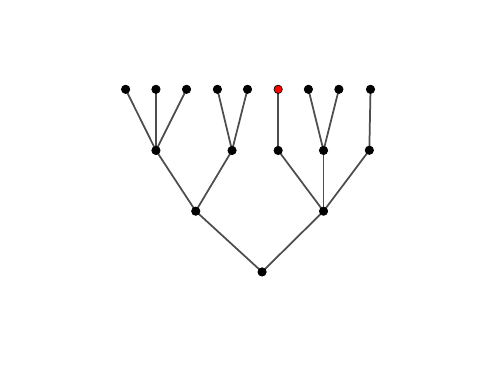}%
\end{center}
\caption{A finite filtration and the corresponding tree with a marked vertex.}%
\label{fig1}%
\end{figure}

Let $\Xi = \{\xi_n\}_{n \geq 0}$ be a (colored) basic Borel
filtration on a space~$X$. For $n \geq 0$ and $x \in X$,
consider the (colored) ordered graded tree $\otp_n(x) \in \OTP$
corresponding to the restriction of the finite filtration
$\{\xi_i\}_{i=0}^n$ to the element of the partition~$\xi_n$
containing the point~$x$, with the marked leaf corresponding
to the point~$x$. By~$\ot_n(x)$ we denote the same (colored)
ordered tree without marked vertex. On the space~$X$ consider
the Borel partition $\bar\xi_n$ into the preimages of points
under the map~$\otp_n$. We say that the sequence~$\bar \Xi$
of thinning partitions $\{\bar\xi_n\}_{n \geq 0}$ is
\textit{associated} with the basic filtration~$\Xi$.

\medskip

\begin{definition}
We say that a (colored) basic Borel filtration~$\Xi$ on the space~$X$ is \emph{combinatorially definite} if for any two points $x,y \in X$ there exists an index~$n$ such that $x$ and~$y$ lie in different elements of the partition~$\bar\xi_n$.
\end{definition}

As in the metric case (see~\cite{VZ18}), every basic Borel filtration of a separable metric space can be colored so as to become combinatorially definite.

\begin{proposition}\label{pro2}
Let $\Xi = \{\xi_n\}_{n \geq 0}$ be a basic Borel filtration of a separable metric space~$X$. Then each partition can be equipped with a color
\textup(see Definition~{\rm\ref{defcolor}\textup)} so that the resulting colored filtration is combinatorially definite.
\end{proposition}

The proof reproduces that from~\cite{VZ18}, so we only outline its scheme: it suffices to choose an arbitrary sequence $\{\eta_n\}_{n\geq 0}$ of finite Borel partitions of~$X$ that separates the points, and for each $n \geq 0$ define a coloring of the elements of the partition~$\xi_n$ using the partition~$\eta_n$.

In the next subsection, we will show how, given a combinatorially
definite colored filtration, one can construct its adic model.

\subsection{An adic realization of a combinatorially definite
filtration}

The construction suggested in~\cite{VZ18} allows one to realize combinatorially definite (colored) basic filtrations as tail filtrations on graded graphs endowed with an adic structure. It can be carried over to the Borel case without changes. Let us briefly recall this construction. Given a (colored) basic Borel filtration $\Xi=\{\xi_n\}_{n\geq 0}$ of a space~$X$, we construct a graded graph $\Gamma=\Gamma[\Xi]$ as follows. Its $n$th level contains the vertices corresponding to different (colored) trees $\ot_n(x)$, $x \in X$; there are finitely many of them, since the filtration is locally finite (and the set of colors is finite for every~$n$). Two vertices of neighboring levels are joined by an edge if the corresponding (colored) ordered trees are nested. An order on the edges entering every vertex is determined by the order in the tree corresponding to this vertex. If a vertex~$v$ of level~$n$ in~$\Gamma$  corresponds to a tree $\ot \in \OT$, then the paths coming to~$v$ from the vertex of level~$0$ correspond in a natural way to the leaves in the tree~$\ot$ (taking into account the order). The space~$X$ can be mapped to the path space~$\T(\Gamma)$ of the constructed graph: a point $x \in X$ goes to the path passing through the vertices corresponding to the trees $\ot_n(x)$, $n \geq 0$; the beginning of length~$n$ of this path corresponds to the marked leaf in the tree~$\otp_n(x)$.

As in the metric case, we obtain the following result.

\begin{proposition}\label{pro1}
If a \textup(colored\textup) basic Borel filtration~$\Xi$ of a space~$X$ is combinatorially definite\textup, then it is isomorphic to the tail filtration of the constructed graded graph~$\Gamma[\Xi]$ with the adic order.
\end{proposition}

Proposition~\ref{pro1} reduces the problem of finding an adic
realization of a Borel automorphism to the problem of constructing
a combinatorially definite colored basic Borel filtration for this automorphism.

\section{A Borel version of Rokhlin's lemma; constructing a basic
filtration for a Borel automorphism}

Let $T$ be a Borel automorphism of a standard Borel space~$X$. Let $B \subset X$ be a Borel subset. By~$\orb_T(B)$ we denote the orbit of the set~$B$ under~$T$, and by~$\orb_{T,+}(B)$, the positive semi-orbit (sometimes, the symbol~$T$ in the notation for the orbit will be omitted):
$$
\orb_T(B) = \bigcup_{k \in \mathbb{Z}} T^k B, \qquad \orb_{T,+}(B) = \bigcup_{k\geq 0} T^k B.
$$
The \emph{Rokhlin tower} $\tow[B]$ with base~$B$ is the sequence of pairwise disjoint Borel sets $B_0=B$, $B_k = T B_{k-1}\setminus B_0$, $k \geq 1$. The sets~$B_k$ are called the levels of the tower~$\tow[B]$, and $B_0$ is called the base of~$\tow[B]$. Clearly, the union of all levels of the tower~$\tow[B]$ is the positive semi-orbit of~$B$ under~$T$:
$$
\cup_{k \geq 0} B_k = \orbp(B).
$$
Let $h \in \mathbb{N}$. The \emph{tower~$\tow_h[B]$ of height~$h$ with base~$B$} is the part of the Rokhlin tower~$\tow[B]$ defined above consisting of the levels $B_0, \dots, B_{h-1}$. We say that a tower is \emph{full} if every  its level  is the full image of the base under the corresponding power of~$T$, i.e.,  $B_k = T^k B_0$.  By $\bar\tow_h[B]$ we denote the union of all levels of the tower~$\tow_h[B]$. Saying that towers $\tow_{h_1}[B_1]$ and $\tow_{h_2}[B_2]$ are disjoint, we mean that the corresponding sets   $\bar\tow_{h_1}[B_1]$ and $\bar\tow_{h_2}[B_2]$ are disjoint.

\subsection{A weakened Borel version of Rokhlin's lemma}

The classical Rokhlin's lemma underlies the theory of uniform approximation of automorphisms of a Lebesgue space. The problem of finding a Borel version of the lemma was posed by V.~A.~Rokhlin in a conversation with the first author. Glasner and Weiss (see~\cite[Proposition~7.9]{GV06}) proved the following Borel analog of the lemma.

\begin{lemma}\label{lemGV}
Let $T$ be a homeomorphism of a Polish space~$(X, \rho)$. Let $\eps>0$\textup, $n \in \mathbb{N}$. Then there exists a Borel subset
 $B\subset X$ such that the tower $\tow_n[B]$
is full and $\mu(\bar\tow_h[B])>1-\eps$ for every measure~$\mu \in
\mapp(X,T)$.
\end{lemma}

It will be convenient for us to modify this statement so as to make
it more suitable for iteration and construction of a basic filtration.

\begin{definition}
A \emph{signature} is a finite nonempty subset of the set of positive integers: $\si = \{h_1,\dots,h_n\}\subset \mathbb{N}$. We say that a signature $\si$ is \emph{primitive} if the numbers from~$\si$ are jointly relatively prime. The signature $\si=\{1\}$ will be called \emph{trivial}.

Let $\xi$ be a partition of a space~$X$. We say that $\xi$ \emph{is subordinate to a signature~$\si$} if every element of~$\xi$ is finite and its cardinality is contained in~$\si$.
\end{definition}

\begin{lemma}\label{lem1}
Let $T$ be a Borel automorphism of a separable metric space~$(X,\rho)$. Let $\si=\{h_1,\dots,h_n\}$ be a nontrivial primitive signature. Then for every $\eps>0$ there exist Borel subsets $B_1, \dots, B_n \subset X$ such that the following properties hold\textup:
\begin{enumerate}
\item[\rm (1)] the towers $\tow_{h_i}[B_i]$\textup, $i=1,\dots,n$\textup, are full and pairwise disjoint\textup;
\item[\rm (2)] $T$ is a bijection on $\bigcup\limits_{i=1}^n\bar\tow_{h_i}[B_i]$\textup;
\item[\rm (3)] for every measure $\mu \in \mapp(X,T)$\textup,
$$
\sum_{i=1}^n \mu\Big(\bar\tow_{h_i}[B_i]\Big)=1, \qquad \sum_{i=1}^n \mu(B_i)<\frac12+\eps.
$$
\end{enumerate}

\end{lemma}
\begin{proof}
We may assume without loss of generality that
$1 \leq h_1<\dots<h_n$. Take a positive integer~$N$ such that every positive integer   $m\geq N$ can be represented as a sum $m=\sum\limits_{i=1}^n\alpha_i h_i$ with nonnegative integer coefficients~$\alpha_i$ with the additional constraint $\alpha_1h_1< \eps m$. Consider the set
$$
X_N=\{x\in X \colon T^k x \ne x \quad \text{for}\quad 1\leq k\leq N \}.
$$
We want to find a sequence of Borel sets $\{U_j\}_{j \in \mathbb{N}}$ such that for every  $j$ the tower~$\tow_N[U_j]$ is full and $X_N = \cup_{j} U_j$.

We may assume without loss of generality that the metric~$\rho$ is bounded. Consider a new metric on~$X$:
$$
\widetilde    \rho(x,y) = \sum_{i=0}^\infty 2^{-i}\rho(T^i x,T^i y), \quad x,y \in X.
$$
This series is absolutely convergent, and the metric~$\widetilde   \rho$  is separable on~$X$. It is easy to see that the map~$T$ is Lipschitz in this metric:
$\widetilde   \rho(Tx,Ty)\leq 2 \widetilde   \rho(x,y)$. For every point $x \in X_N$, the points $T^i x$, $i=0,\dots, N$, are pairwise distinct, hence there exists $\delta>0$ such that the open balls $B_{\widetilde   \rho}(T^ix, 2^i \delta)$, $i=0,\dots, N$, are pairwise disjoint. Then the shifts of the open ball $B_{\widetilde   \rho}(x, \delta)$ under the transformations $T^i$, $i=0,\dots,N$, are pairwise disjoint. Since the space~$(X,\widetilde   \rho)$ is separable, we can represent the set~$X_N$ as a countable union of such balls. Let us call them $U_j$, $j \in \mathbb{N}$. It remains to show that every such open ball in the metric~$\widetilde    \rho$ is measurable with respect to the Borel $\sigma$-algebra generated by the metric~$\rho$. For  fixed $x$, for every~$i$, the function $y \mapsto \rho(T^i x,T^i y)$ is Borel measurable as the composition of Borel measurable maps. Therefore, the function $y\mapsto \widetilde   \rho(x,y)$ is also Borel measurable, which implies the measurability of the ball in the metric~$\widetilde    \rho$. So, a desired family~$U_j$ is constructed.

Now let us construct a sequence of Borel sets~$\{V_j\}$ such that the semi-orbits~$\orbp(V_j)$ are pairwise disjoint and
$$
\bigcup_j \orb(V_j) =\bigcup_j \orb(U_j) = X_N.
$$
This sequence can be defined recursively: $V_1= U_1$, and
$$
V_j = U_j \setminus \Big(\bigcup_{l=1}^{j-1} \orb(V_l)\Big),  \quad j \geq 2.
$$

For fixed $j \in \mathbb{N}$, let $V_{j,k}$ be the $k$th  level  in the tower~$\tow[V_j]$. By construction, $V_j \subset U_j$, hence the tower~$\tow_N[V_j]$ is full. Each of the sets~$V_j$ can be represented as a countable union:
\begin{align*}
V_j = \widetilde    V_{j,\infty} \cup \bigcup_{k \geq N} \widetilde    V_{j,k},& \quad  \text{ where } \widetilde    V_{j,k} = T^{-k}V_{j,k} \setminus T^{-k-1}V_{j,k+1}, \notag
\\
\widetilde    V_{j,\infty} &= \bigcap_{k \geq N} T^{-k}V_{j,k}.
\end{align*}
Thus, $\widetilde    V_{j,k}$ is the set of points of the tower base over which there are exactly $k$ levels, and $\widetilde    V_{j,\infty}$ is the set of points over which the tower is infinite; hence, the towers $\tow_{k}[\widetilde    V_{j,k}]$ and $\tow[\widetilde    V_{j,\infty}]$ are full. We will call them elementary towers. Every elementary tower has height at least~$N$, hence it can be represented as a disjoint union of full towers of height $h_1,\dots,h_n$; moreover, a representation can be chosen in such a way that the total portion of levels of  towers of height~$h_1$ does not exceed~$\eps$ (and when decomposing elementary towers of infinite height, one can do without towers of height~$h_1$ at all). Combining the towers of the same height~$h_i$,  $i =1,\dots, n$, into one, denote its base by~$A_i$.

Clearly, all sets involved in the construction are Borel, and the towers $\tow_{h_i}[A_i]$, $i =1,\dots, n$, are pairwise disjoint. For a measure $\mu \in \mapp(X,T)$, we have $\mu(X_N)=1$. Note that the union of the constructed towers coincides with the union of the semi-orbits~$\orbp(V_j)$. Obviously, for every~$k \geq 1$ and every~$j$, the set $T^{-k}V_j \setminus \orbp(V_j)$ has pairwise disjoint images under the transformations $T^{lk}$, $l \geq 1$, hence it has zero measure. It follows that
$$
\mu\Big(\bigcup_{i=1}^n \bar\tow_{h_i}[A_i]\Big)=
\mu\Big(\bigcup \orbp(V_j)\Big)=1.$$
It is clear from construction that for every finite elementary tower $\tow_{k}[\widetilde    V_{j,k}]$,
$$
\mu\Big(\bar\tow_{h_1}[A_1] \cap \bar\tow_{k}[\widetilde    V_{j,k}]\Big)\leq \eps \mu\Big(\bar\tow_{k}[\widetilde    V_{j,k}]\Big),
$$
hence $\mu\big(\bar\tow_{h_1}[A_1]\big)\leq\eps$. Since the signature~$\si$ is not trivial, $h_2 \geq 2$. Since all towers $\tow_{h_i}[A_i],$ $i=1,\dots, n$, are full, this implies the inequality
$$
\sum_{i=1}^n \mu(A_i) = \mu(A_1) + \sum_{i=2}^n \mu(A_i) \leq \mu(A_1) + \sum_{i=2}^n \frac12 \mu\big(\bar\tow_{h_i}[A_i]\big)\leq \frac{1+\eps}{2}< \frac12 +\eps.
$$

Let $A = \bigcup\limits_{i=1}^n \bar\tow_{h_i}[A_i]$ be the union of all constructed towers. It is clear from construction that $T(A)\subset A$. Put $B=\bigcap\limits_{k\geq 0} T^{-k} A$ and $B_i=A_i \cap B$, $i=1,\dots, n$. Obviously, the sets~$B_i$ inherit the properties of the sets~$A_i$ verified earlier, but now the map~$T$ is a bijection from the set  $\bigcup\limits_{i=1}^n \bar\tow_{h_i}[B_i]$ onto itself.
\end{proof}

\noindent{\bf~3.2. Constructing a basic filtration with a
given signature}

\begin{theorem}\label{th1}
Let $\{\si_k\}_{k \in \mathbb{N}}$ be a sequence of nontrivial
primitive signatures. Let  $T$ be a Borel automorphism of a
separable metric space $(X,\rho)$. Then there is a
metrically universal  $T$\nobreakdash-in\-va\-riant subset $\widehat  X\subset X$ consisting of aperiodic points and a basic Borel filtration $\Xi=\{\xi_k\}_{k \geq 0}$ on the set~$\widehat  X$ such that the partition $\xi_{k}/\xi_{k-1}$ is subordinate to the signature~$\si_{k}$ for all~$k \geq 1$.
\end{theorem}

\begin{proof}
The proof is based on an iterative application of Lemma~\ref{lem1}.
Fix a sequence $\eps_n$, $n \in \mathbb{N}$, of positive numbers
converging to zero. For convenience, we will assume that
$(X_1,\rho_1)=(X,\rho)$, $T_1=T$.

Let us describe one step of the construction. Let $m \geq 1$, and let $T_m$ be a Borel automorphism of the separable metric space $(X_m,\rho_m)$.
Let $\si_m=\{h_{1,m},\dots,h_{n,m}\}$, where $h_{1,m}\!<\dots<h_{n,m}$ and ${n=n(m)}$. Apply Lemma~\ref{lem1} to the Borel automorphism~$T_m$ of the space $(X_m,\rho_m)$, the signature~$\si_m$, and the number~$\eps_m$. Find the corresponding collection of Borel sets $B_{1,m},\dots, B_{n,m}$. Put $X_{m+1}=\bigcup\limits_{j=1}^n B_{j,m}$ (the union of the tower bases) and $Y_m =\bigcup\limits_{j=1}^n \bar\tow_{h_{j,m}}[B_{j,m}]\!=\!\orb_{T_m}(X_{m+1})$  (the union of the towers themselves). Define a projection $\pi_m\colon Y_m \to X_{m+1}$ of points of the tower to its base:
$$
\pi_m(T_m^k x) = x\quad \text{for } x \in B_{j,m}, \quad k=0,\dots, h_{j,m}-1, \quad j=1,\dots,n.
$$
Since $T_m$ sends Borel sets to Borel sets, it is clear
that the map~$\pi_m$ is Borel. Define a partition~$\nu_m$ on~$Y_m$ as the partition into the preimages of points under the projection~$\pi_m$. In other words, elements of~$\nu_m$ are sets of the form $\{T_m^kx\colon k=0,\dots, h_{j,m}-1\}$, $j=1,\dots, n$, $x \in B_{j,m}$.

Define a map~$T_{m+1}$ on~$X_{m+1}$ as a recurrence map: namely,
for every point $x \in X_{m+1}$, put $T_{m+1}(x) = T_m^{r_m(x)}x,$
where $r_m(x) = \min\{k \colon k>0, T_m^kx\in X_{m+1}\}$.
One can easily see that for $j=1,\dots, n$, for $x \in B_{j,m}$,
we have $r_m(x)=h_{j,m}$, since, by construction, the set~$Y_m$
(the union of the towers) is invariant under~$T_m$. One can easily check that $T_{m+1}$ is a Borel automorphism of the separable metric space $(X_{m+1}, \rho_{m+1})$, where $\rho_{m+1}$ is the restriction of the metric~$\rho_m$ to~$X_{m+1}$.

Put
$$\widehat  X  = \bigcap_{m \geq 1}  \orb_T(X_m) \setminus \orb_T\Big(\bigcap_{m \geq 1}  X_m\Big).
$$
One can easily show, by induction on~$m \geq 1$, that for every point $x \in \widehat  X$ the intersection $\orb_{T}(x)\cap X_m$ consists of a single $T_m$-orbit, i.e., for every point $y \in \orb_{T}(x)\cap X_m$ we have $\orb_{T}(x)\cap X_m =  \orb_{T_m}(y)\cap X_m$. It follows that for $m\geq 1$ we have $\widehat  X\cap X_m = \widehat  X \cap Y_m$, hence
$$
\pi_m^{-1} \big(X_{m+1} \cap \widehat  X \big)=Y_m \cap \widehat  X = X_m \cap \widehat  X.
$$

For $m \geq 1$, we define a partition~$\xi_m$ of the set $\widehat  X=X_1\cap \widehat  X$ as the partition into the preimages of points of the set $X_{m+1} \cap \widehat  X$  under the map $\pi_m\circ \dots \circ \pi_1$. Clearly, for every point $x \in X_{m+1} \cap \widehat  X$ the preimage  $\pi_1^{-1}\circ \dots \circ \pi_m^{-1}(x)$ coincides with the part $\{T^k x \colon 0 \leq k<  N\}$ of the orbit $\orb_T(x)$, where $N$ is such that $T^N x = T_{m+1} x$. One can easily see that the partition $\xi_{m+1}$ is coarser than~$\xi_m$, and that the quotient of~$\xi_{m+1}$ by~$\xi_{m}$ is isomorphic to the restriction of the partition~$\nu_{m+1}$ to $Y_{m+1}\cap \widehat  X$ and hence is subordinate to the signature~$\si_{m+1}$.

Let us check that the set~$\widehat  X$  is metrically universal.
Given a measure  $\mu \in \mapp(X,T)$, let $\mu_m$, $m \geq 1$,
be the restriction of~$\mu$ to the set~$X_m$. Obviously, the measure~$\mu_m$ is invariant under~$T_m$. Points from~$X_m$ that are periodic for~$T_m$ are also periodic for~$T$, hence the measures~$\mu_m$ are aperiodic. Lemma~\ref{lem1} guarantees that $\mu_m(X_m\setminus Y_m)=0$. It follows that
$$
\mu(\orb_T(X_m)\setminus \orb_T(X_{m+1}))=0,
$$
and now one can easily show by induction that  $\mu(\orb_T(X_m))=1$ for every~$m$. Lemma~\ref{lem1} guarantees that $\mu(X_{m+1})<(\frac{1}{2}+\eps_m)\mu(X_m)$, hence $\mu(X_m) \to 0$, which implies that
$$
\mu\Big(\orb_T\Big(\bigcap\limits_{m \geq 1}  X_m\Big)\Big)=0.
$$
 Thus, $\mu(\widehat  X)=1.$

It remains to check that the constructed filtration is basic. It suffices to show that for $x \in \widehat  X$ the points~$x$ and~$T^{-1}x$ lie in the same element of the partition~$\xi_m$ for sufficiently large~$m$. But for every~$m$ the points for which $x$ and~$T^{-1}x$ do not lie in the same element of~$\xi_m$ are exactly the points of the set $X_{m+1}\cap \widehat  X$. These sets are nested, and their intersection over all $m \geq 1$ is empty.
\end{proof}

\section{The uniadic graph}\label{sec2}
Recall the construction of the uniadic (universal semidyadic)
graph (see~\cite{VZ18}). Level~$0$ of this graph contains
a single vertex. Further, having the set~$V_n$ of vertices
of level~$n$, we define the set~$V_{n+1}$ of vertices of
level~$n+1$ as $V_{n+1} = V_n^2 \sqcup \mathrm{copy}(V_n)$.
Every vertex~$w$ from~$V_n^2$ is understood as an ordered pair~$(u,v)$
of vertices of level~$n$, and we draw edges from~$u$ and~$v$ to~$w$
endowing them with a natural order: the edge~$(v,w)$ is greater
than~$(u,w)$. Every vertex~$w$ from~$\mathrm{copy}(V_n)$ is
understood as a copy of some vertex~$u$ of level~$n$,
and we draw a unique edge from the vertex~$u$ to the vertex~$w$.
The resulting graph endowed with the adic structure will be
called  \emph{uniadic} and denoted by~${\rm\uag}$ (see Fig.~2).
The term ``uniadic'' derives from the words ``universal''
and ``semi-dyadic,'' where the latter means that every vertex
of level~$n$, for $n \geq 1$, has one or two edges coming to it from vertices of level~$n-1$. The predecessors of this graph are dyadic graphs: the \emph{graph of unordered pairs} (see~\cite{V17}) and the \emph{graph of ordered pairs} (see~\cite{VZ17});
each of them is of considerable interest.

\begin{figure}[h!]%
\begin{center}
\includegraphics[width=0.6\columnwidth]{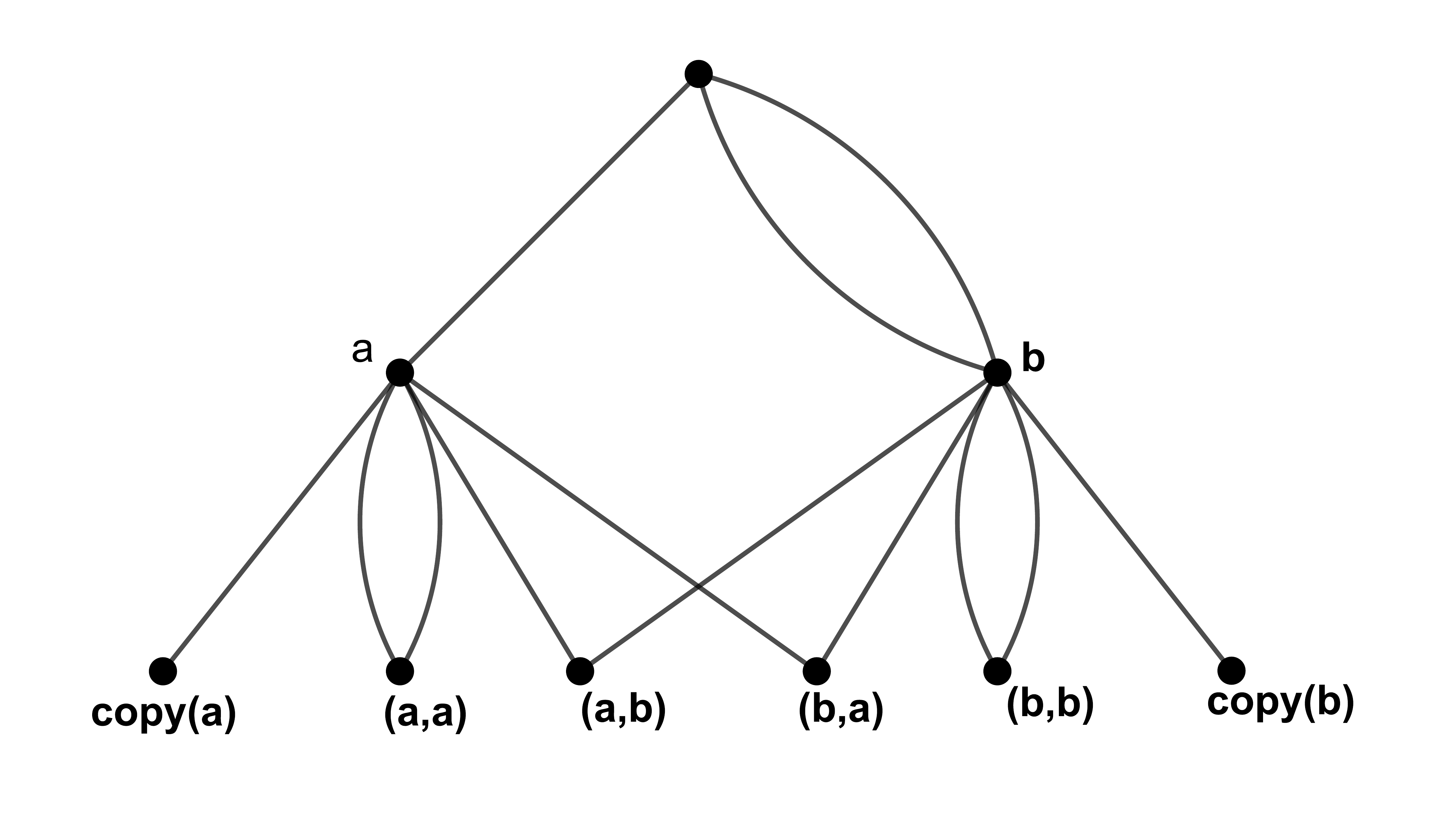}%
\end{center}
\vspace{-20pt}
\caption{Several first levels of the uniadic graph~\uag.}%
\label{figuag}%
\end{figure}

Recall two constructions applied to graded graphs: induction
and telescoping.

\begin{definition}
Let $\Gamma$ be a graded graph and $\{k_n\}_{n\geq 0}$ be a strictly
increasing sequence of nonnegative integers with  $k_0=0$.
We define a \emph{telescoping} of the graph~$\Gamma$ as follows.
The $n$th level of the new graph contains the vertices corresponding to the vertices of level~$k_n$ in~$\Gamma$. A vertex of level~$n$ and a vertex of level~$n+1$ are joined by a multiple edge with multiplicity equal to the number of paths in~$\Gamma$ between the corresponding vertices. An adic order on the edges of the new graph is determined by the adic order on the corresponding paths in the original graph.
\end{definition}

One can easily see that the adic shifts on the path spaces of the original graph and the telescoped graph are isomorphic.

\begin{definition}
We say that a graded graph~$\Gamma_1$ is an \emph{induced subgraph} of a graded graph~$\Gamma$ if the set of vertices and the set of edges of~$\Gamma_1$ are subsets in the set of vertices and the set of edges of~$\Gamma$, respectively, and, besides, if $v$ is a vertex of~$\Gamma_1$, then $\Gamma_1$ contains all edges of~$\Gamma$ coming to~$v$ from vertices of the previous level. An order on the edges is inherited in a natural way.
\end{definition}

If $\Gamma_1$ is an induced subgraph of~$\Gamma$, then its path space~$\T(\Gamma_1)$ is a subset of the path space~$\T(\Gamma)$ of~$\Gamma$ invariant under the adic shift in~$\Gamma$.

To prove the main theorem, we will use the following proposition proved in~\cite{VZ18}.

\begin{proposition}\label{pro3}
Let $\Gamma$ be a graded graph with an adic structure whose
every vertex, except the vertex of level~$0$, has at least two
edges coming to it from above. Then there is an induced subgraph
of the uniadic graph~${\rm\uag}$ such that a telescoping of this subgraph is isomorphic to~$\Gamma$.
\end{proposition}

\section{Proof of Theorem~\ref{thborel}}
To prove the theorem, it remains to apply the obtained results. Let $T$ be an aperiodic Borel automorphism of a separable metric space~$X$. Choose the following sequence of signatures: $\si_k=\{2,3\}$ for $k \geq 1$. Further, applying Theorem~\ref{th1}, find a metrically universal subset $\widetilde    X \subset X$ and a basic filtration $\Xi=\{\xi_n\}_{n\geq 0}$ on~$\widetilde    X$ subordinate to the signature. Using Proposition~\ref{pro2}, color the filtration so as to make it combinatorially definite. Proposition~\ref{pro1} shows that the colored basic filtration~$\Xi$ is isomorphic to the tail filtration of the graded graph $\Gamma=\Gamma[\Xi]$ with the adic shift. Note that every vertex of~$\Gamma$, except the vertex of level~$0$, has either two or three edges coming to it from above: this follows immediately from the construction of the graph and the fact that the filtration~$\Xi$ is subordinate to the chosen signature. It remains to apply the last ingredient of the proof, Proposition~\ref{pro3}. It allows us to embed the path space of the graph~$\Gamma[\Xi]$  equivariantly into the path space of the uniadic graph.

\end{document}